\documentclass[11pt]{article}
\RequirePackage{latexsym}
\RequirePackage{array}
\RequirePackage{amsmath,amssymb,amsthm,amsfonts}        % AMS stuff
\RequirePackage{vmargin}
\RequirePackage{ifthen}                                 % Latex control structures
\RequirePackage{xspace}                                 % Useful for adding spacing
\usepackage{times}
\usepackage{hyperref}

\setpapersize[portrait]{USletter}
\setmarginsrb{1.1in}{0.9in}{1.1in}{0.7in}{0in}{0in}{12pt}{36pt}
\newcommand{\thmabove}{7.5pt}
\newcommand{\thmbelow}{0pt}
\newcommand{\proofbelow}{0pt}
\let\latexcite=\cite
\def\cite{\nolinebreak\latexcite}
\let\latexref=\ref
\def\ref{\nolinebreak\latexref}

\newcommand{\ClaimName}[1]{\label{clm:#1}}

\newcommand{\EquationName}[1]{\label{eq:#1}\text{}}

\newcommand{\LemmaName}[1]{\label{lem:#1}}

\newcommand{\SectionName}[1]{\label{sec:#1}}

\newcommand{\Claim}[1]{Claim~\ref{clm:#1}}

\newcommand{\Equation}[1]{Eq.\:\eqref{eq:#1}}

\newcommand{\Lemma}[1]{Lemma~\ref{lem:#1}}

\newcommand{\Section}[1]{Section~\ref{sec:#1}}

    \newtheoremstyle{mythmstyle}
      {\thmabove}   % Space above
      {\thmbelow}   % Space below
      {}            % Font of theorem body (e.g., \itshape)
      {}            % Indent amount (empty = no indent, \parindent = para indent)
      {\bfseries}   % Thm head font
      {. }          % Punctuation after thm head
      {2.5pt}       % Space after thm head (\newline = linebreak)
      {\thmname{#1}\thmnumber{ #2}\thmnote{ \normalfont (#3)}}   % Thm head spec
    \theoremstyle{mythmstyle}
    \newtheorem{theorem}{Theorem}
    
    \newtheorem{claim}[theorem]{Claim}

    \newtheorem{lemma}[theorem]{Lemma}

\newcommand{\afterproof}{\hfill $\blacksquare$ \par \vspace{\proofbelow}}

\renewenvironment{proof}{\noindent\textbf{Proof.}\,}{\afterproof}
\newenvironment{proofof}[1]{\noindent\textbf{Proof} \,(of #1).\,}{\afterproof}

\renewcommand{\th}{\ifmmode{^{\textrm{th}}}\else{\textsuperscript{th}\ }\fi}

\allowdisplaybreaks[3]      % Allow page breaks in the middle of displayed equations

\newcommand{\set}[1]{\left \{ #1 \right \}}                     % Set notation: { ... }
\newcommand{\setst}[2]{\left\{\; #1 \,:\, #2 \;\right\}}        % Set notation: { ... | ... }
\newcommand{\sumstack}[1]{\sum_{\substack{#1}}}
\newcommand{\union}{\cup}                                       

\newcommand{\intersect}{\cap}

\newcommand{\Ind}{\operatorname{Ind}}
\newcommand{\qdown}{\breve{q}}
\newcommand{\medcap}{{\textstyle \bigcap}}
\newcommand{\Pdown}{\breve{P}}

\newcommand{\InclNbr}{\Gamma^+}

\renewcommand{\Pr}{\mathbb{P}}

\begin{document}

\title{Short proofs for generalizations of the Lov\'asz Local Lemma: \\
Shearer's condition and cluster expansion}
\author{Nicholas J. A. Harvey \and Jan Vondr\'ak}
\date{}
\maketitle

\begin{abstract}
The Lov\'asz Local Lemma is a seminal result in probabilistic combinatorics.
It gives a sufficient condition on a probability space and a collection of events
for the existence of an outcome that simultaneously avoids all of those events.
Over the years, more general conditions have been discovered under which
the conclusion of the lemma continues to hold.
In this note we provide short proofs of two of those more general results:
Shearer's lemma and the cluster expansion lemma, in their ``lopsided'' form.
We conclude by using the cluster expansion lemma to prove that
the symmetric form of the local lemma holds with probabilities bounded by $1/ed$,
rather than the bound $1/e(d+1)$ required by the traditional proofs.
\end{abstract}

%%%%%%%%%%%%%%%%%%%%%%%%%%%%%%%%%%%%%%%%%%%%%%%%%%%%%%%%%%%%%%%%%%%%%%%%%%%%%

\section{Introduction}

The Lov\'{a}sz Local Lemma, due to Erd\H{o}s and Lov\'{a}sz \cite{ErdosLovasz},
has become a very useful tool for showing the existence of objects with special properties.
It is particularly important in combinatorics, theoretical computer science,
and also in number theory.

The statement of the LLL involves a probability space on $\Omega$ with events $E_1,\ldots,E_n$.
The desired conclusion is that $\Pr[ \cap_i \overline{E_i} ] > 0$.
The hypothesis is that there exists a graph on vertex set $[n] = \set{1,\ldots,n}$ such that
\begin{enumerate}
\item[\rm (i)] non-neighbors in the graph have, in a certain sense, limited dependence, and
\item[\rm(ii)] the probabilities of the events must satisfy a certain upper bound.
\end{enumerate}
In the original formulation of the LLL \cite{ErdosLovasz}, 
condition (i) is that each event must be independent from its non-neigbors,
and condition (ii) is that each event must have probability at most $1/4d$, where $d$ 
is the maximum degree in the graph.

Over the years, new formulations of condition (i) were discovered,
of which a very general one is stated below as inequality \eqref{eq:Dep}.
Instead of requiring independence between non-neighbors,
it allows arbitrary dependencies, as long as one can establish a useful upper bound on the
probability of $E_i$ conditioned on any set of its non-neighboring events not occurring.
We believe this condition first appeared in a paper by Albert, Frieze and Reed \cite{Albert},
and is sometimes referred to as the ``lopsided" version of the LLL.
(This is more general than the condition used by Erdos and Spencer~\cite{ErdosSpencer}.)

Several new formulations of condition (ii) have been proposed over the years,
notably by Spencer \cite{Spencer75,Spencer77} and by Shearer \cite{Shearer}.
Shearer's condition is actually optimal, assuming that the graph is undirected.
Unfortunately Shearer's condition is difficult to use in applications, so
researchers have also studied weaker conditions that are easier to use.
One of the most useful of those is the ``cluster expansion'' condition, due to Bissacot
et~al.~\cite{Bissacot}.

In this note, we present short, self-contained proofs of the LLL 
in which condition (i) is formalized using \eqref{eq:Dep}, as in Albert et~al.~\cite{Albert},
and condition (ii) is formalized using either Shearer's condition \cite{Shearer}
or the cluster expansion condition \cite{Bissacot}.
\Section{ShearerShort} gives a short proof for Shearer's condition.
Our proof follows the line of Shearer's original argument,
although we believe our exposition is simpler and more direct.
\Section{cluster} gives a short proof for the cluster expansion condition.
Whereas Bissacot et al.\ used analytic methods inspired 
by statistical physics, we found a short combinatorial inductive argument. 
This combinatorial proof originally appeared in Section 5.7 of \cite{HV-arxiv}, 
but since that may be somewhat difficult to find, we reproduce it here.
To conclude, we show that the cluster expansion condition implies the
near-optimal $p \leq \frac{1}{ed}$ condition for the symmetric LLL.

%%%%%%%%%%%%%%%%%%%%%%%%%%%%%%%%%%%%%%%%%%%%%%%%%%%%%%%%%%%%%%%%%%%%%%%%%%%%%%%%
\section{Shearer's Lemma}
\SectionName{ShearerShort}

The following result is the ``lopsided Shearer's Lemma", a generalization of the LLL
combining conditions from Albert et~al.~\cite{Albert} and Shearer \cite{Shearer}.
This formulation also appears in \cite{Knuth}.
Let $\Gamma(i)$ denote the neighbors of vertex $i$ and let $\InclNbr(i) = \Gamma(i) \cup \set{i}$.
Let $\Ind=\Ind(G)$ denote the collection of all independent sets in the graph $G$.

\begin{lemma}[lopsided Shearer's Lemma]
\LemmaName{extShearer}
Suppose that $G$ is a graph and $E_1,\ldots,E_n$ events such that
\begin{equation}
\EquationName{Dep}
\Pr[E_i \mid \medcap_{j \in J} \overline{E_j}] ~\leq~ p_i
\qquad\forall i \in [n] ,\, J \subseteq [n] \setminus \InclNbr(i).
\end{equation}
For each $S \subseteq [n]$, define
$$\qdown_S ~=~ \qdown_S(p) ~=~ \sumstack{I \subseteq S \\ I \in \Ind(G)} (-1)^{|I|} \prod_{i \in I} p_i.$$
If $\qdown_S \geq 0$ for all $S \subseteq [n]$, then for each $A \subseteq [n]$, we have
$$ \Pr[\medcap_{j \in A} \overline{E_j}] \geq \qdown_A.$$
\end{lemma}

We present an inductive proof of this lemma. First, we state the following recursive identity for $\qdown_A$.

\begin{claim}[The ``fundamental identity'' for $\qdown$.] % {\protect Shearer \cite{Shearer}, Scott-Sokal \cite[Eq.~(3.5)]{ScottSokal}}]
\ClaimName{fundamental-q}
For any $a \in A$, we have
$$\qdown_A ~=~ \qdown_{A \setminus \set{a}} \,-\, p_a \cdot \qdown_{A \setminus \InclNbr(a)}.$$
\end{claim}

\begin{proof}
Every independent set $I \subseteq A$ either contains $a$ or does not. In addition, if $a \in I$
then $I$ is independent iff $I \setminus \{a\}$ is an independent subset of $A \setminus \InclNbr(a)$. Thus the terms in $\qdown_A$ correspond one-to-one to terms on the right-hand side.
\end{proof}

Next. define $\Pdown_A = \Pr[\bigcap_{i \in A} \overline{E_i}]$.
The following claim analogous to \Claim{fundamental-q} is the key inequality in the original proof
of the LLL \cite{ErdosLovasz,Spencer75,AlonSpencer} although typical expositions do not
call attention to it.

\begin{claim}[The ``fundamental inequality" for $\Pdown$]
\ClaimName{fundamentalP}
Assume that \eqref{eq:Dep} holds.
Then for each $a \in A$,
$$\Pdown_A \geq \Pdown_{A - a} - p_a \Pdown_{A \setminus \InclNbr(a)}.$$
\end{claim}

\begin{proof}
The claim is derived as follows.
$$
\Pdown_A 
 ~=~ \Pdown_{A-a} - \Pr\Bigg[ E_a \cap \bigcap_{i \in A-a} \overline{E_i} \Bigg]
 ~\geq~ \Pdown_{A-a} - \Pr\Bigg[ E_a \cap \bigcap_{i \in A \setminus \InclNbr(a)} \overline{E_i} \Bigg]
 ~\geq~ \Pdown_{A-a} - p_a \Pdown_{A \setminus \InclNbr(a)}
$$
The first inequality is trivial (by monotonicity of measure with respect to taking subsets) and the
second inequality is our assumption with $J = A \setminus \InclNbr(a)$.
\end{proof}

Given these two claims, Shearer's Lemma follows by induction.  \vspace{3pt}

\begin{proofof}{\Lemma{extShearer}}
We claim by induction on $|A|$ that for all $a \in A$, 
\begin{equation}
\EquationName{ShearerInduction}
\frac{\breve{P}_{A}}{\breve{P}_{A-a}} \geq \frac{\breve{q}_{A}}{\breve{q}_{A-a}}.
\end{equation}
The base case, $A = \{a\}$, holds because $\breve{P}_{\{a\}} = \breve{q}_{\{a\}} = 1 - p_a$
and $\breve{P}_\emptyset = \breve{q}_\emptyset = 1$.

For $|A|>1$, the inductive hypothesis applied successively to the elements of $A \cap \Gamma(a)$ yields
$$ \frac{\breve{P}_{A-a}}{\breve{P}_{A \setminus \InclNbr(a)}} \geq
\frac{\breve{q}_{A-a}}{\breve{q}_{A \setminus \InclNbr(a)}}.$$
Using this inequality, \Claim{fundamental-q} and \Claim{fundamentalP}, we obtain
$$ \frac{\breve{P}_{A}}{\breve{P}_{A-a}} \geq 1 - p_a \frac{\breve{P}_{A \setminus
\InclNbr(a)}}{\breve{P}_{A-a}} \geq 1 - p_a \frac{\breve{q}_{A \setminus \InclNbr(a)}}{\breve{q}_{A-a}} 
= \frac{\breve{q}_A}{\breve{q}_{A-a}}.$$
This proves the inductive claim.

Combining \eqref{eq:ShearerInduction} with the fact $\breve{P}_\emptyset = \breve{q}_\emptyset = 1$
shows that $\breve{P}_A \geq \breve{q}_A$ for all $A$.
\end{proofof}

\paragraph{Comparison to Shearer's original lemma.}
Shearer's lemma was originally stated as follows \cite{Shearer}.

\begin{lemma}
\LemmaName{origShearer}
Suppose that
\begin{equation}
\EquationName{Dep2}
\Pr[E_i \mid \medcap_{j \in J} \overline{E_j}] ~=~ \Pr[E_i]
\qquad\forall i \in [n] ,\, J \subseteq [n] \setminus \InclNbr(i).
\end{equation}
Let $p_i = \Pr[E_i]$ and for each $S \subseteq [n]$, define
$$q_S = \sumstack{I \in \Ind(G) \\ S \subseteq I} (-1)^{|I \setminus S|} \prod_{i \in I} p_i.$$
If $q_S \geq 0$ for all $S \subseteq [n]$, then
$$ \Pr[\medcap_{i=1}^{n} \overline{E_i}] \geq q_\emptyset.$$
\end{lemma}

There are two differences between \Lemma{extShearer} and \Lemma{origShearer}.
Regarding condition (i), \Lemma{extShearer} uses \eqref{eq:Dep}
whereas \Lemma{origShearer} uses \eqref{eq:Dep2};
as discussed above, the former condition is more general.
The other main difference is the use of coefficients $q_S$ in \Lemma{origShearer} as opposed to $\qdown_S$ in \Lemma{extShearer}. 
The condition $q_S \geq 0 \:\forall S$ turns out to be equivalent to $\qdown_S \geq 0 \:\forall S$,
so \Lemma{extShearer} and \Lemma{origShearer} have equivalent formulations of condition (ii) (see \cite{HV-arxiv} for more details).
We chose to state \Lemma{extShearer} using $\qdown_S$ because those are the coefficients that
naturally arise in the proof.

Shearer gives an interpretation of these coefficients as follows:
there is a unique probability space called the ``tight instance" that minimizes the
probability of $\Pr[ \medcap_i \overline{E_i} ]$.
In that probability space, $q_S$ is exactly the probability that the events $\setst{ E_i }{ i \in S }$
occur and the events $\setst{ E_j }{ j \notin S }$ do not occur.
In contrast, the coefficient $\qdown_S$ is the probability that the events $\setst{ E_i }{ i \in S }$
do not occur. 
In general, the coefficients are related by the identity $\qdown_S = \sum_{T \subseteq [n] \setminus
S} q_T$, which can be proved by inclusion-exclusion.
The conclusion of \Lemma{extShearer} is that $\Pr[\bigcap_{i=1}^{n} \overline{E_i}] \geq \qdown_{[n]}$ and it is easy to see that $\qdown_{[n]} = q_\emptyset$. 
Hence we recover \Lemma{origShearer} from \Lemma{extShearer}.
The tight instance also shows that the conclusion of Shearer's lemma is tight.

%%%%%%%%%%%%%%%%%%%%%%%%%%%%%%%%%%%%%%%%%%%%%%%%%%%%%%%%%%%%%%%%%%%%%%%%%%%%%%%%
\section{Cluster Expansion}
\label{sec:cluster}

Next we turn to a variant of the LLL that is stronger than the early formulations
\cite{ErdosLovasz,Spencer75,Spencer77} but weaker than Shearer's Lemma.
This lemma has been referred to as the {\em cluster expansion} variant of the LLL; it was proved by Bissacot et al.~\cite{Bissacot} using analytic techniques inspired by statistical physics. Although it is subsumed by Shearer's Lemma, it is typically easier to use in applications and
provides stronger quantitative results than the original LLL.

For variables $y_1,\ldots,y_n$, we define
$$
Y_S ~=~ \sumstack{I \in \Ind \\ I \subseteq S} y^I,
$$
where $y^I$ denotes $\prod_{i \in I} y_i$.
This is similar to the quantity $\qdown_S$, but without the alternating sign.

\begin{lemma}[the cluster expansion lemma]
\LemmaName{cluster}
Suppose that
\eqref{eq:Dep} holds 
and there exist $y_1,\ldots,y_n>0$ such that for each $i \in [n]$,
\begin{equation}
\EquationName{CLL}
p_i ~\leq~ \frac{y_i}{Y_{\InclNbr(i)}}.
\end{equation}
Then
$$ \Pr[\medcap_{i=1}^{n} \overline{E_i}] ~\geq~ \frac{1}{Y_{[n]}} ~>~ 0.$$
\end{lemma}

Here we present an inductive combinatorial proof of \Lemma{cluster}. 
First, some preliminary facts.

\begin{claim}[The ``Fundamental Identity'' for $Y$]
\ClaimName{fundamentalY}
For any $a \in A$, we have
$$ Y_{A} = Y_{A - a} + y_a Y_{A \setminus \InclNbr(a)}.$$
\end{claim}

\begin{proof}
This follows from \Claim{fundamental-q} since we can write $Y_S = \qdown_S(-y)$.
Or directly, every summand $y^J$ on the left-hand side either appears in $Y_{A - a}$
if $a \not\in J$, or as a summand in $y_a Y_{A \setminus \InclNbr(a)}$ if $a \in J$.
\end{proof}

\begin{claim}[Log-subadditivity of $Y$]
\ClaimName{submult}
If $A, B$ are disjoint then $Y_{A \union B} \leq Y_A \cdot Y_B$.
\end{claim}
\begin{proof}
Every summand $y^J$ of $Y_{A \cup B}$ appears in the expansion of the product
$$
Y_A \cdot Y_B = 
 \sumstack{J' \subseteq A \\ J' \in \Ind} 
 \sumstack{J'' \subseteq B \\ J'' \in \Ind} y^{J'} y^{J''}
$$
by taking $J' = J \intersect A$ and $J'' = J \intersect B$.
All other terms on the right-hand side are non-negative.
\end{proof}

The following is the key inductive inequality, analogous to \Equation{ShearerInduction}
in the proof of Shearer's Lemma. Note that here the induction runs
in the opposite direction for the $Y$ coefficients, which are indexed by complementary sets;
the reason for this lies in the fundamental identity for $Y$ (\Claim{fundamentalY}) which
has the opposite sign compared to \Claim{fundamental-q}.
For a set $S \subseteq [n]$,  we will use the notation $S^c = [n] \setminus S$.

\begin{lemma}
\LemmaName{cluster-induction}
Suppose that $p$ satisfies \eqref{eq:CLL}.
Then for every $a \in S \subseteq [n]$, $\Pdown_S > 0$ and
$$
\frac{\Pdown_{S}}{\Pdown_{S-a}} ~\geq~ \frac{Y_{S^c}}{Y_{(S-a)^c}}.
$$
\end{lemma}

\begin{proof}
First, note that \eqref{eq:CLL} implies that $p_i < 1$ for all $i$.
We proceed by induction on $|S|$. The base case is $S = \{a\}$. In that case we have
$ \frac{\Pdown_{\{a\}}}{\Pdown_\emptyset} = \Pr[\overline{E_a}] \geq 1 - p_a > 0$.
On the other hand, by the two claims above and \eqref{eq:CLL}, we have 
$$ Y_{[n]} ~=~ Y_{[n]-a} + y_a Y_{[n] \setminus \InclNbr(a)}
 ~\geq~ Y_{[n] - a} + p_a Y_{\InclNbr(a)} Y_{[n] \setminus \InclNbr(a)}
 ~\geq~ Y_{[n] - a} + p_a Y_{[n]}.$$
Therefore, $\frac{Y_{[n]-a}}{Y_{[n]}} \leq 1 - p_a$ which proves the base case.

We prove the inductive step by similar manipulations.  Let $a \in S$.
We can assume that $\Pdown_{S-a} > 0$ by the inductive hypothesis.
By \Claim{fundamentalP}, we have
$$ \frac{\Pdown_S}{\Pdown_{S-a}} ~\ge~ 
1 - p_a \frac{\Pdown_{S \setminus \InclNbr(a)}}{\Pdown_{S-a}}.$$
The inductive hypothesis applied repeatedly to the elements of $S \cap \Gamma(a)$ yields
\begin{equation*}
1 - p_a \frac{\Pdown_{S \setminus \InclNbr(a)}}{\Pdown_{S-a}}
 ~\geq~ 1 - p_a \frac{Y_{(S \setminus \InclNbr(a))^c}}{Y_{(S-a)^c}}
 ~=~ 1 - p_a \frac{Y_{S^c \cup \InclNbr(a)}}{Y_{S^c + a}}. 
\end{equation*}
By the two claims above and \eqref{eq:CLL}, we have
$$ Y_{S^c+a} ~=~ Y_{S^c} + y_a Y_{S^c \setminus \InclNbr(a)}
 ~\geq~ Y_{S^c} + p_a Y_{\InclNbr(a)} Y_{S^c \setminus \InclNbr(a)}
 ~\geq~ Y_{S^c} + p_a Y_{S^c \cup \InclNbr(a)}.$$
We conclude that
$$ \frac{\Pdown_{S}}{\Pdown_{S-a}}
 ~\geq~ 1 - p_a \frac{Y_{S^c \cup \InclNbr(a)}}{Y_{S^c+a}}
 ~\geq~ 1 - \frac{Y_{S^c+a} - Y_{S^c}}{Y_{S^c+a}} = \frac{Y_{S^c}}{Y_{(S-a)^c}} $$
which also implies $\Pdown_S > 0$.
\end{proof}

Now we can complete the proof of \Lemma{cluster}.

\begin{proofof}{\Lemma{cluster}}
By \Lemma{cluster-induction}, we have $\frac{\Pdown_{S}}{\Pdown_{S-a}} ~\geq~ \frac{Y_{S^c}}{Y_{(S-a)^c}}$ for all $a \in S$.
Hence,
$$ \Pr[\bigcap_{i=1}^{n} \overline{E_i}] ~=~ \Pdown_{[n]}
 ~=~ \prod_{i=1}^{n} \frac{\Pdown_{[i]}}{\Pdown_{[i-1]}}
 ~\geq~ \prod_{i=1}^n \frac{Y_{[i]^c}}{Y_{[i-1]^c}}
 ~=~ \frac{1}{Y_{[n]}}.
$$
\end{proofof}

By a similar proof, it can be proved that $\frac{\qdown_{S}}{\qdown_{S-a}} ~\geq~ \frac{Y_{S^c}}{Y_{(S-a)^c}}$ for all $a \in S$,
which relates the cluster expansion lemma to Shearer's Lemma. We refer the reader to \cite{HV-arxiv}.

%%%%%%%%%%%%%%%%%%%%%%%%%%%%%%%%%%%%%%%%%%%%%%%%%%%%%%%%%%%%%%%%%%%%%%%%%%%%%%%%
\section{The Symmetric LLL}

The ``symmetric'' LLL does not use a different upper bound $p_i$
for each event $E_i$, and instead assigns $p_i = p$ for all $i$.
The question then becomes, given a dependency graph,
what is the maximum $p$ such that the conclusion of the LLL holds?
As mentioned above, Erd\H{o}s and Lov\'{a}sz \cite{ErdosLovasz} show that one may take $p=1/4d$
if the graph has maximum degree $d$.
Spencer \cite{Spencer77} showed the improved result $p = d^d/(d+1)^{d+1} > \frac{1}{e(d+1)}$,
and Shearer \cite{Shearer} improved that to the value
$p = (d-1)^{d-1}/d^d > \frac{1}{ed}$, which is optimal as $n \rightarrow \infty$.

We now show that the cluster expansion lemma (\Lemma{cluster}) gives a short proof
of the $\frac{1}{ed}$ bound, which is just slightly suboptimal.
Alternative proofs of the $(d-1)^{d-1}/d^d$ and $\frac{1}{ed}$ bounds may be found in
Knuth's exercises 323 and 325 \cite{Knuth}.

\begin{lemma}[Near-optimal symmetric LLL]
Suppose that $G$ has maximum degree $d \geq 2$ and let
$p = \max_{i \in [n]} \, \max_{J \subseteq [n] \setminus \InclNbr(i)} \,
\Pr[E_i \mid \medcap_{j \in J} \overline{E_j}]$.
If $$p \leq \frac{1}{ed}$$
then $$\Pr[ \medcap_i \overline{E_i} ] > 0.$$
\end{lemma}
\begin{proof}
We set $p_i=p$ and $y_i=y=\frac{1}{d-1}$ for all $i$,
then apply \Lemma{cluster}.
To do so, we must check that \eqref{eq:CLL} is satisfied.
Note that $Y_{\InclNbr(i)} = y + Y_{\Gamma(i)} \leq y + (1+y)^d$.
Then
$$
\frac{y_i}{Y_{\InclNbr(i)}}
    \geq \frac{y}{y+(1+y)^d}
    %= \frac{\frac{1}{d-1}}{\frac{1}{d-1}+\big(1+\frac{1}{d-1}\big)^d}
       = \frac{1}{1+\frac{d^d}{(d-1)^{d-1}}}.
$$
The claim is that this is at least $\frac{1}{ed}$.
By simple manipulations, this claim is equivalent to
\begin{equation}
\EquationName{SymmIneq}
e ~\geq~ \frac{1}{d} + \Big( \frac{d}{d-1} \Big)^{d-1},
\end{equation}
which we prove by a short calculus argument.
First we derive the bound
\begin{align*}
\ln \Big( \frac{d}{d-1} \Big)^{d-1}
\,=\, \!- (d-1) \ln \Big( 1 - \frac{1}{d} \Big)
\,=\, (d-1) \sum_{k=1}^{\infty} \frac{1}{k d^k}
\,=\, 1 - \sum_{k=1}^{\infty} \Big(\frac{1}{k}-\frac{1}{k+1}\Big) \frac{1}{d^k}
\,<\, 1 - \frac{1}{2d}.
\end{align*}
From here, we obtain
$$
\Big( \frac{d}{d-1} \Big)^{d-1}
~<~ \exp\Big( 1 - \frac{1}{2d} \Big)
~<~ e \cdot \Big( 1 - \frac{1}{2d} \Big)
~<~ e - \frac{1}{d}
$$
which establishes \eqref{eq:SymmIneq}.
\end{proof}

\bibliographystyle{plain}
\bibliography{LLL}

\end{document}